\documentclass[10pt,a4paper]{article}
\usepackage{graphicx}
\usepackage{amsmath}
\usepackage{amsfonts}
\usepackage{amssymb}
\usepackage{amsthm}
\usepackage[margin=1in]{geometry}
\usepackage{color}
\begin{document}

\def\R{\mathbb{R}}
\def\N{\mathbb{N}}
\def\H{\mathcal{H}}
\def\d{\textrm{div}}
\def\v{\textbf{v}}
\def\I{\hat{I}}
\def\B{\hat{B}}
\def\x{\hat{x}}
\def\y{\hat{y}}
\def\p{\hat{\phi}}
\def\P{\mathcal{P}}
\def\r{\hat{r}}
\def\w{\textbf{w}}
\def\u{\textbf{u}}
\def\K{\mathcal{K}}
\def\Reg{\textrm{Reg}}
\def\s{\textrm{Sing}}
\def\sgn{\textrm{sgn}}

\numberwithin{equation}{section}

\newtheorem{lem}{Lemma}[section]
\newtheorem{rem}{Remark}[section]
\newtheorem{cor}{Corollary}[section]
\newtheorem{thm}{Theorem}[section]
\newtheorem{prop}{Proposition}[section]
\newtheorem{definition}{Definition}[section]
\newtheorem{con}{Conjecture}[section]
\newtheorem{Main}{Main Result}

\title{On the $L^\infty-$maximization of the solution of Poisson's equation: Brezis-Gallouet-Wainger type inequalities and applications}

\author{Davit Harutyunyan\thanks{University of California Santa Barbara, harutyunyan@ucsb.edu} and Hayk Mikayelyan\thanks{School of Mathematical Sciences, University of Nottingham Ningbo, 199 Taikang East Road, Ningbo 315100, PR China   \hfill Hayk.Mikayelyan@nottingham.edu.cn}}

\maketitle

\begin{abstract}
For the solution of the Poisson problem with an $L^\infty$ right hand side
\begin{equation*}
\begin{cases} -\Delta u(x) = f (x)  & \mbox{in } D,
 \\
 u=0  & \mbox{on } \partial D,  \end{cases}
\end{equation*}
we derive an optimal estimate of the form
$$
\|u\|_\infty\leq \|f\|_\infty \sigma_D(\|f\|_1/\|f\|_\infty),
$$
where $\sigma_D$ is a modulus of continuity defined in the interval $[0, |D|]$ and depends only on the domain $D$.
In the case when $f\geq 0$ in $D$ the inequality is optimal for any domain and for any values of $\|f\|_1$ and $\|f\|_\infty.$
We also show that
$$
\sigma_D(t)\leq\sigma_B(t),\text{ for }t\in[0,|D|],
$$
where $B$ is a ball and $|B|=|D|$. Using this optimality property of $\sigma,$ we derive Brezis-Galloute-Wainger type inequalities on
the $L^\infty$ norm of $u$ in terms of the $L^1$ and $L^\infty$ norms of $f.$ The estimates have explicit coefficients depending on the space dimension $n$ and turn to equality for a specific choice of $u$ when the domain $D$ is a ball.
As an application we derive $L^\infty-L^1$ estimates on the $k-$th Laplace eigenfunction of the domain $D.$
\end{abstract}


\section{Introduction and problem setting}
\label{sec:1}
The Brezis-Gallouet inequality [\ref{bib:Bre.Gal.}] is a crucial tool in establishing the unique solvability of initial boundary value problem for
the nonlinear Schr\"odinger equation with zero Dirichlet data on a smooth domain in $\mathbb R^2$ as shown by Brezis and Gallouet in [\ref{bib:Bre.Gal.}]. The inequality asserts the following: \textit{For any smooth bounded domain $\Omega,$ there exists a constant $C>0$ depending only on $\Omega,$ such that for any function $u\in H^2(\Omega)$ with $\|u\|_{H^1(\Omega)}\leq 1,$ there holds: }
\begin{equation}
\label{1.1}
\|u\|_{L^\infty(\Omega)}\leq C(\Omega)\left(1+\sqrt{\log(1+\|u\|_{H^2(\Omega)})}\right).
\end{equation}
Brezis and Wainger [\ref{bib:Bre.Wai.}] extended (\ref{1.1}) to higher dimensions proving the following: \textit{Assume $k,l,n\in\mathbb N$ and $q>0$ are such that $1\leq k<l,$ $k\leq n$ and $ql>n.$ Then there exists a constant $C>0$ such that for any $u\in W^{l,q}(\mathbb R^n)$ with
$\|u\|_{W^{k,\frac{n}{k}}(\mathbb R^n)}=1,$ one has the estimate}
\begin{equation}
\label{1.2}
\|u\|_{L^\infty(\mathbb R^n)}\leq C\left(1+\left(\log(1+\|u\|_{W^{l,q}(\mathbb R^n)})\right)^{\frac{n-k}{n}}\right).
\end{equation}
The Brezis-Gallouet-Wainger and similar inequalities and sharp constants in them have been extensively studied in different frameworks in
[\ref{bib:Engler},\ref{bib:Gorka},\ref{bib:Maz.Sha.},\ref{bib:Mor.Sat.Sav.Wad.},\ref{bib:Mor.Sat.Wad.},\ref{bib:Sawano}]. The estimate has been known to be a central tool in the study of several types of partial differential equations [\ref{bib:Gig.Gig.Saa.},\ref{bib:Kre.Maz.}], in particular Navier-Stokes equations and turbulence [\ref{bib:Foi.Man.Ros.Tem.},\ref{bib:Kat.Pon.},\ref{bib:Vishik}].
In the present paper we are concerned with estimating the $L^\infty$ norm of a function $u$ by the $L^1$ and $L^\infty$ norms of the Laplacian
$\Delta u.$ More precisely we are concerned with the following question: Given a bounded connected open Lipschitz set $D\subset\mathbb R^n$ and a number $0\leq \beta\leq |\Omega|,$ let $u_f$ be the unique weak solution of the Dirichlet boundary value problem
\begin{equation}
\label{1.3}
\begin{cases}
-\Delta u_f=f & \mbox{in } D,\\
u_f=0 & \mbox{on } \partial D,\\
\end{cases}
\end{equation}
where the function $f$ belongs to the class
\begin{equation}
\label{1.4}
\mathcal{C}=\{f\colon D\to\mathbb R \ : \ |f(x)|\leq 1\ \ \text{a.e. in} \ D,\  \text{and}\ \  \int_D|f(x)|dx=\beta\}.
\end{equation}
We consider the maximization problem
\begin{equation}
\label{1.5}
\sup_{f\in \mathcal{C}} \Phi(f),
\end{equation}
where
\begin{equation}
\label{1.6}
\Phi(f)=\sup_{x\in D} |u_f(x)|.
\end{equation}
The question is, \textit{what is the function $f$ having an $L^\infty$ norm bounded by $1$ and a fixed $L^1$ norm, that maximizes the $L^\infty$ norm of the solution $u_f$?} As will be seen in the next section, we obtain a Brezis-Gallouet-Wainger type inequality on the $L^\infty$ norm of $u$ in terms of $L^1$ and $L^\infty$ norms of the Laplacian $\Delta u$ in all dimensions. For given $D\subset\mathbb R^n$ and $f\in\mathcal{C}$ with the additional property that $f\geq 0$ in $D,$ we also characterize the points $\hat x\in\Omega,$ where the function $|u|$ attains its maximal value. Note, that in contrast to Brezis and Wainger we do not impose higher Sobolev regularity upon $u,$ but rather we assume that $u\in H_0^1(\Omega)$ such that $\Delta u\in L^\infty.$ The striking novelty here is that in the obtained inequalities the coefficients are explicit in the space dimension $n$ and also they turn into an equality for a specific choice of the function $u,$ when the domain $D$ is a ball. As an application of the new estimates, we also derive $L^\infty$ estimates on the $k-$th Laplace eigenfunction $u_k$ in terms of its $L^1$ norm and the $k-$th eigenvalue $\lambda_k,$ again with explicit coefficients.

Typically working with $L^1$ and $L^\infty$ norms of solutions of partial differential equations is a more delicate task than dealing with $L^p$ norms for $1<p<\infty$. It is well known for instance, that one cannot estimate the $L^\infty$ norm of the solution $u$ by the $L^1$ norm of $f$. The trivial counter-example for the Poisson problem in the ball $B_1$
\begin{equation}
\label{1.7}
\begin{cases}
 -\Delta u(x) = f (x)  & \mbox{in } B_1,
 \\
 u=0  & \mbox{on } \partial B_1,
\end{cases}
\end{equation}
is given by the family of radially symmetric functions
\begin{equation}
\label{1.8}
u_r(|x|)=
\begin{cases}
F(1)-F(|x|)  & \mbox{in } B_1\setminus B_r,
 \\
 F(1)-F(r)+\frac{1}{2r^{n-2}}-\frac{1}{2r^n} |x|^2 & \mbox{in }  B_r,
\end{cases}
\end{equation}
where $F(|x|)$ is the fundamental solution multiplied by a dimensional constant such that $F'(r)= \frac{1}{r^{n-1}}$. Then all the functions $f_r(x)=-\Delta u_r(x)=\frac{n}{r^n}\chi_{B_r}(x)$, $0<r<1$, have the same $L^1$ norm, while the family of solutions $u_r$ are unbounded in $L^\infty.$



\section{Main results}
\label{sec:2}
In what follows all the norms will be over the domain $D$ unless otherwise specified. Thus for the sake of convenience we will leave out the domain $D$ from the norm notation $\|v\|_{L^p(D)}$ simply using the notation $\|v\|_{p}$ when there is no ambiguity.
We consider first the case when $f$ is nonnegative. For that purpose denote
\begin{equation}
\label{2.1}
\mathcal{C^{+}}=\{f\colon D\to\mathbb R \ : \ 0\leq f(x)\leq 1\ \ \text{a.e. in} \ D,\  \text{and}\ \  \int_Df(x)dx=\beta\}.
\end{equation}
We have the following existence and characterization theorem.
\begin{thm}[Nonnegative Laplacian]
\label{th:2.1}
The maximization problem
\begin{equation}
\label{2.2}
\sup_{f\in\mathcal{C^{+}}}\Phi(f),
\end{equation}
has a solution $\hat{f}\in\mathcal{C^{+}}$. For a maximizer $\hat{f}$ there exists a unique point $\hat{x}$ such that
\begin{equation}
\label{2.3}
\hat{f}(x)=\chi_{\{y\,:\,G(\hat{x},y)>\alpha_{\hat{x}}\}}(x)\in \mathcal{C^{+}},
\end{equation}
where $G(x,y)$ is the Green's function and $\alpha_{\hat{x}}$ is chosen to satisfy
\begin{equation}
\label{2.4}
| \{y\,:\,G(\hat{x},y)>\alpha_{\hat{x}}\}|=\beta.
\end{equation}
Moreover, for the solution $-\Delta \hat{u}=\hat{f},$ $\hat{u}\in H_0^1(D),$ one has the maximality inequality
\begin{equation}
\label{2.5}
\hat{u}(x)\leq \hat{u}(\hat{x})\quad\text{ for any }\quad x\in D,
\end{equation}
and the point $\hat{x}$ satisfies the condition
\begin{equation}
\label{2.6}
0=\nabla \hat{u}(\hat{x})=\int_D \nabla_{x} G(\hat{x}, y) \chi_{\{G(\hat{x}, y)>\alpha_{\hat{x}}\}}dy.
\end{equation}
\end{thm}

\begin{thm}[Modulus of continuity]
\label{th:2.2}
For any $f\in L^\infty$, such that $\|f\|_\infty>0$ we have the inequality
\begin{equation}
\label{2.7}
\|u\|_\infty\leq \|f\|_\infty \sigma_D(\|f\|_1/\|f\|_\infty),
\end{equation}
where $\sigma_D$ is a modulus of continuity defined in the interval $[0, |D|]$ and depending only on the domain $D.$
Moreover, for any domain $D$ and values of $\|f\|_1$ and $\|f\|_\infty$ there exists a function $f\in\mathcal{C^+}$, such that the inequality
(\ref{2.7}) turns to an equality.
\end{thm}

Next we give an optimality condition on $\sigma_D.$

\begin{thm}[Estimates on $\sigma_D$ and $\|u\|_{\infty}$]
\label{th:2.3}
We have the optimality estimate
\begin{equation}
\label{2.8}
\sigma_D(t)\leq\sigma_B(t),\quad\text{ for }\quad t\in[0,|D|],
\end{equation}
where $B$ is a ball with the same measure as $D,$ i.e., $|B|=|D|.$ The function $\sigma_B$ can be calculated explicitly and has the form
\begin{align}
\label{2.9}
\sigma_{B}(t)&=\frac{(n-1)^2+1}{2n(n-2)\omega_n^{2/n}}t^{\frac{2}{n}}-\frac{1}{n(n-2)\omega_nR^{n-2}}t,\quad \text{for}\quad n>2,\\ \nonumber
\sigma_{B}(t)&=\left(\frac{1}{2}\ln\pi +\frac{1}{2\pi}(1+\ln R)\right)t-\frac{1}{4\pi} t\ln t,\quad \text{for}\quad n=2,
\end{align}
where $R$ is the radius of $B$ and $\omega_n$ is the volume of the unit ball in $\mathbb R^n.$ For the norm $\|u\|_{\infty}$ we have the estimates
\begin{align}
\label{2.10}
\|u_f\|_{\infty}& \leq \frac{(n-1)^2+1}{2n(n-2)\omega_n^{2/n}}\|f\|_{1}^{\frac{2}{n}}\|f\|_{\infty}^{\frac{n-2}{n}}-
\frac{1}{n(n-2)\omega_nR^{n-2}}\|f\|_{1},\quad \text{for}\quad n>2 \\ \nonumber
\|u_f\|_{\infty} &\leq \left(\frac{1}{2}\ln\pi +\frac{1}{2\pi}(1+\ln R)\right)\|f\|_{1}-
\frac{1}{4\pi}\|f\|_{1}\ln\frac{\|f\|_{1}}{\|f\|_{\infty}},\quad \text{for}\quad n=2,
\end{align}
here, as already mentioned, $R=\left(\frac{|D|}{\omega_n}\right)^{\frac{1}{n}}.$
\end{thm}

For the general case we have the following theorem which follows from Theroem~\ref{th:2.3}.

\begin{thm}[Sign-changing Laplacian]
\label{th:2.4}
The solution $u_f\in H_0^1(D)$ of the Poisson problem
$$-\Delta u=f\quad\text{in}\quad D,$$
where $f\in L^\infty(D),$ fulfills the following inequalities:
\begin{equation}
\label{2.11}
\|u_f\|_\infty \leq\max \left\{\|f^+\|_\infty \sigma_D(\|f^+\|_1/\|f^+\|_\infty)\,\, ,\,\, \|f^-\|_\infty \sigma_D(\|f^-\|_1/\|f^-\|_\infty)    \right\},
\end{equation}
where $f^{\pm}=\max (\pm f, 0)$. Also we have the bound
\begin{equation}
\label{signchanging}
\|u\|_\infty \leq \|f\|_\infty \left[ \sigma_D\left( \frac{1}{2}\left[\tfrac{1}{\|f\|_\infty}\left|I_f\right| + |D| \right]\right)-  v(\hat{x})\right],
\end{equation}
where the function $v\in H_0^1(D)$ is the solution to the Poisson problem $-\Delta v=1$ in $D$, the point $\hat{x}\in D$ is the point where the maximum of $|u|$ is achieved, and $I_f=\int_D f(x)dx.$
Moreover, one has the same estimates as in (2.10):
\begin{align}
\label{2.12}
\|u_f\|_{\infty}& \leq \frac{(n-1)^2+1}{2n(n-2)\omega_n^{2/n}}\|f\|_{1}^{\frac{2}{n}}\|f\|_{\infty}^{\frac{n-2}{n}}-
\frac{1}{n(n-2)\omega_nR^{n-2}}\|f\|_{1},\quad \text{for}\quad n>2 \\ \nonumber
\|u_f\|_{\infty} &\leq \left(\frac{1}{2}\ln\pi +\frac{1}{2\pi}(1+\ln R)\right)\|f\|_{1}-
\frac{1}{4\pi}\|f\|_{1}\ln\frac{\|f\|_{1}}{\|f\|_{\infty}},\quad \text{for}\quad n=2.
\end{align}
\end{thm}

\subsection{An application to the Laplace eigenfunctions}
\label{sec:2.1}

As a consequence of the estimates (\ref{2.12}), we derive an estimate on the $L^\infty$ norm of the $k-$th eigenfunction of the Laplace operator in the domain $D.$ Assume as usual that $(\lambda_k,u_k)$ is the $k-$th eigenvalue-eigenfunction pair of the Laplace operator, i.e.,
\begin{equation}
\label{2.13}
\begin{cases}
-\Delta u_k=\lambda_k u_k & \mbox{ in } \ \ D,\\
u_k=0 & \mbox{ on } \ \ \partial D.\\
\end{cases}
\end{equation}
We have the following theorem.
\begin{thm}
\label{th:2.5}
The following estimates hold:
\begin{align}
\label{2.14}
\|u_k\|_{\infty}& \leq \frac{2}{n^n(n-2)\omega_n}\left[\lambda_k^{\frac{n}{2}}\left((n-1)^2+1\right)^\frac{n}{2}-\frac{\lambda_kn^{n-1}}{R^{n-2}}\right]\|u_k\|_1,
\quad \text{for}\quad n>2 \\ \nonumber
\|u_k\|_{\infty} &\leq \lambda_k\left(\ln\pi+\frac{1}{\pi}(1+\ln R)+\frac{\lambda_k}{8\pi^2}\right)\|u_k\|_1,\quad \text{for}\quad n=2.
\end{align}
\end{thm}


\section{Proof of the main results}
\label{sec:3}

\begin{proof}[Proof of Theorem~\ref{th:2.1}]
The proof is divided into several steps. Of course, in the first step we will be proving the existence part of the theorem.\\
\textbf{Existence of a maximizer $\hat{f}.$} We start with collecting some useful bounds. First of all note that by the estimate $0\leq f\leq 1$ for $f\in\mathcal{C^+}$ and the maximum principle [\ref{bib:Gil.Tru.}, Theorem~8.19], one has on one hand that the weak solution $u_f\in H_0^1(D)$ of $-\Delta u=f$ in $D$ is nonnegative in $D,$ i.e.,
\begin{equation}
\label{3.1}
u_f(x)\geq 0,\quad\text{for all}\quad f\in\mathcal{C^+}, x\in D.
\end{equation}
On the other hand, if $\bar u\in H_0^1(D)$ is the unique weak solution of $-\Delta \bar u=1$ in $D,$ then we have $-\Delta (\bar u-u_f)=1-f\geq 0$ in $D$ for any $f\in\mathcal{C^+},$ and $\bar u-u_f\in H_0^1(D),$ thus again by the maximum principle we have
\begin{equation}
\label{3.2}
0\leq u_f(x)\leq \bar u(x),\quad\text{for all}\quad f\in\mathcal{C^+}, x\in D.
\end{equation}
By the classical regularity theory we have $\bar u\in C(\bar D),$ thus $\|\bar u\|_{\infty}<\infty.$ Combining the obtained bounds we arrive at
\begin{equation}
\label{3.3}
0\leq u_f(x)\leq \bar u(x)\leq \|\bar u\|_{\infty}<\infty,\quad\text{for all}\quad f\in\mathcal{C^{+}}, x\in D.
\end{equation}
We adopt the direct method in the calculus of variations, i.e., choose a maximizing sequence $f_k\in\mathcal{C^{+}}$ for the problem
(\ref{2.2}), such that
\begin{equation}
\label{3.4}
\Phi(f_k)>\sup_{f\in\mathcal{C^{+}}}\Phi(f)-\frac{1}{k}=m-\frac{1}{k},
\end{equation}
where we clearly have due to (\ref{3.3}) the bound
$$m=\sup_{f\in\mathcal{C^{+}}}\Phi(f)\leq M=\|\bar u\|_{\infty}<\infty.$$
As the sequence $\{f_k\}$ is bounded in $L^\infty(D),$ then by weak$^{\ast}$ compactness, there exists a function $f_0\in L^\infty(D),$ such that
\begin{equation}
\label{3.5}
f_k\rightharpoonup f_0 \quad\text{in}\quad L^\infty(D),
\end{equation}
for a subsequence (not relabeled). By the convergence of the averages
$$\int_E f_k\to \int_E f_0,$$
for any measurable subset $E\subset D,$ we have that $0\leq \int_E f_0\leq |E|,$ thus we get $0\leq f_0(x)\leq 1$ a.e. $x\in D.$ Also
it follows that $\int_E f_0=\beta,$ thus $f_0\in \mathcal{C^+}.$ We aim to prove that the unique $u_0\in H_0^1(D)$ solution of $-\Delta u_0=f_0$ in $D$ satisfies the condition $\sup_{x\in D}|u_0(x)|=m.$ Let $u_k\in H_0^1(D)$ be the unique weak solution of the equation $-\Delta u_k=f_k.$ Then by the estimate $\|u_k\|_{H^1}\leq C\|f\|_{2},$ and weak $L^2$ compactness, there exists a function $u\in H_0^1(D)$ such that for a subsequence (not relabeled) we have
\begin{equation}
\label{3.6}
u_k\to u\quad\text{in} \quad L^2(D),\quad\text{and}\quad \nabla u_k\rightharpoonup \nabla u\quad\text{weakly in} \quad L^2(D).
\end{equation}
For any $\varphi\in C^1(\bar D)$ we have
$$\int_D f_0\varphi=\lim_{k\to\infty}\int_D f_k\varphi=\lim_{k\to\infty}\int_D \nabla u_k\nabla\varphi=\int_D \nabla u\nabla\varphi,$$
thus the function $u$ solves the equation $-\Delta u=f_0$ in $D,$ therefore we have $u=u_0.$
Note that as $\bar u\in C(\bar D)$ and $\bar u=0$ on $\partial D,$ then there exists $\delta>0$ such that
$|\bar u(x)|<\frac{m}{2}$ if $x\in D$ with $\mathrm{dist}(x,\partial D)\leq \delta.$ Taking into account (\ref{3.2}), we obtain the estimate
 \begin{equation}
\label{3.7}
0< u_k(x) <\frac{m}{2}\quad\text{if}\quad x\in D,\ \ \mathrm{dist}(x,\partial D)\leq \delta.
\end{equation}
Recall next the following classical local H\"older regularity result [\ref{bib:Gil.Tru.}, Theorem~8.22], which we formulate below for the convenience of the reader (the below formulation is the simplified version for the Laplace operator of the original version of Theorem~8.22 in [\ref{bib:Gil.Tru.}]).
\begin{thm}
\label{th:3.1}
Assume $g\in L^{\frac{q}{2}}(D)$ for some $q>n.$ Then if $v\in H^1(D)$ solves $-\Delta v=g$ in $D,$ it follows that $v$ is locally H\"older continuous in $\Omega,$ and for any $y\in D,$ any $0<R\leq R_0$ such that $B_{R_0}\subset D,$ one has the estimate
\begin{equation}
\label{3.8}
\underset{B_R(y)}{\mathrm{osc}} v\leq CR^\gamma (R_0^{-\gamma}\underset{B_{R_0}(y)}{\mathrm{sup}}|v|+\|g\|_{L^q(B_{R_0})}),
\end{equation}
where the constants $C,\gamma>0$ depend only on $n,q$ and $R_0$ and
$\underset{E}{\mathrm{osc}}\ v=\underset{E}{\mathrm{sup}}\ v-\underset{E}{\mathrm{inf}}\ v$ for any $E\subset D.$
\end{thm}
Owing to (\ref{3.4}), we choose a point $x_k\in D$ with the property that $|u_k(x_k)|>m-\frac{1}{k}.$ Then (\ref{3.7}) implies that for big enough $k$ one has $x_k\in D_\delta,$ where $D_\delta=\{x\in D \ : \ \mathrm{dist}(x,\partial D)\geq \delta\}.$ We apply Theorem~3.1 to the function
$v=u_k-u_0$ (where we take $q=2n$) in the ball $B_{\delta}(x_k)\subset D,$ to get
\begin{equation}
\label{3.9}
\underset{B_R(x_k)}{\mathrm{osc}} (u_k-u_0)\leq C_\delta R^\gamma (\delta^{-\gamma}\underset{B_{\delta}(x_k)}{\mathrm{sup}}|u_k-u_0|+\|f_k-f_0\|_{L^{2n}(B_{\delta})}),
\end{equation}
for any $0<R\leq\delta.$ We have that $\|f_k-f_0\|_{L^{2n}(B_{\delta})}\leq \|f_k-f_0\|_{L^{2n}(D)}\leq 2|D|^{\frac{1}{2n}}$ and also by (\ref{3.3}) that $\underset{B_{\delta}(x_k)}{\mathrm{sup}}|u_k-u_0|\leq 2M,$ thus we get from (\ref{3.9}) the bound
\begin{equation}
\label{3.10}
\underset{B_R(x_k)}{\mathrm{osc}} (u_k-u_0)\leq C R^\gamma,
\end{equation}
for some constant $C$ depending only on $n,\delta,|D|$ and $M.$ Denote next $\epsilon_k=|u_k(x_k)-u_0(x_k)|.$ The goal is now to prove that $\epsilon_k\to 0$ as $k\to\infty.$ Assume by contradiction that for some subsequence (not relabeled) one has $\epsilon_k\geq \epsilon>0.$ Then by (\ref{3.10}) there exists a constant radius $0<r<\delta$ such that $|u_k(x)-u_0(x)|\geq \frac{\epsilon}{2}$ for $x\in B_r(x_k),$ which implies the bound
$$\|u_k-u_0\|_{L^2(D)}^2\geq \int_{B_r(x_k)}|u_k(x)-u_0(x)|^2dx\geq c_nr^n\epsilon^2,$$
where $c_n>0$ is a constant depending only on $n.$ The last lower bound contradicts the $L^2$ convergence $u_k\to u_0$ in (\ref{3.6}). The convergence $\epsilon_k\to 0$ and the inequality $|u_k(x_k)|>m-\frac{1}{k}$ give
$$\sup_{x\in D}u_0(x)\geq \liminf_{k\to\infty}u_0(x_k)\geq m,$$
which completes the proof of the existence part.\\
\textbf{Proof of (\ref{2.3})-(\ref{2.5}).} We will utilize the following well-known argument called the bathtub principle [\ref{bib:Lie.Los.}, Theorem 1.14], which is formulated below.
\begin{thm}
\label{th:3.2}
Assume $D\subset\mathbb R^n$ is open bounded connected and Lipschitz and assume $0\leq\beta\leq |D|.$ Let $G\in L^1(D)$ be a non-negative function without flat parts, i.e.,
$$
|\{x\, : \, G(x)=l   \}|=0\quad \text{ for all }\quad l>0.
$$
Then the maximization problem
$$
I=\sup_{g\in \mathcal{C^{+}}} \int_D G(x)g(x)dx
$$
is solved by
$$
\hat{g}(x)=\chi_{\{x\, :\, G(x)>\alpha\}},
$$
where $\alpha$ is chosen such that
$$
|{\{x\, :\, G(x)>\alpha\}}|=\beta.
$$
\end{thm}

Denote now $f_0=\hat{f}$ and $u_0=\hat{u}.$ By the local H\"older continuity of the function $\hat{u}$ (Theorem~3.1) and the property that the maximal value is taken a positive distance away from the boundary (\ref{3.7}), we have that $\sup_D|\hat{u}(x)|$ is attained at a point $\hat{x}\in D_\delta.$ Using Poisson integral formula we have that
$$
\hat{u}(\hat{x})=\int_D G(\hat{x},y)\hat{f}(y)dy\geq \int_D G(\hat{x},y)f(y)dy=u_f(\hat{x})
$$
for any $f\in\mathcal{C^+}.$ Consequently, the bathtub principle give the result
\begin{equation}
\label{2.3555}
\hat{f}(x)=\chi_{\{y\,:\,G(\hat{x},y)>\alpha_{\hat{x}}\}}(x)\in \mathcal{C^{+}}.
\end{equation}
Let us now observe that the set $\{x\,\, : \,\,  \hat{u}(x)=\hat{u}_{max} \}$ consists of one element.
Assume in contradiction $\hat{x},\tilde{x}\in\{x\,\, : \,\,  \hat{u}(x)=\hat{u}_{max} \}$ for some $\hat{x}\neq\tilde{x}.$ Then by (\ref{2.3555}) we have
$$
\hat{f}(x)=\chi_{\{y\,:\,G(\hat{x},y)>\alpha_{\hat{x}}\}}(x)=\chi_{\{y\,:\,G(\tilde{x},y)>\alpha_{\tilde{x}}\}}(x),
$$
thus
$$
G(\hat{x},y)-\frac{\alpha_{\hat{x}}}{\alpha_{\tilde{x}}}G(\tilde{x},y)\equiv 0,
$$
for $y\in D\setminus \{y\,:\,G(\hat{x},y)>\alpha_{\hat{x}}\}$. This implies $\hat{x}=\tilde{x}$ and $\alpha_{\hat{x}}=\alpha_{\tilde{x}}$, which is a contradiction.

\textbf{Proof of (\ref{2.6}).}
The optimality condition
$$
0=\nabla u(\hat{x})=\int_D \nabla_{x} G(\hat{x}, y) \chi_{\{G(\hat{x}, y)>\alpha_{\hat{x}}\}}dy.
$$
follows from $\hat{x}\in \{x\,\, : \,\,  \hat{u}(x)=\hat{u}_{max} \}$.
The theorem is now proven.
\end{proof}


\begin{proof}[Proof of Theorem~\ref{th:2.2}]
Without loss of generality we can assume that $\|f\|_\infty=1$, the general case following from linearity. If $f\in \mathcal{C}^{+},$ then
the proof follows from Theorem~\ref{2.1} with
\begin{equation}
\label{sigmadef}
\sigma(\|f\|_1)=u(\hat{x})=\int_D  G(\hat{x}, y) \chi_{\{G(\hat{x}, y)>\alpha_{\hat{x}}\}}dy.
\end{equation}
In the case when $f\in\mathcal{C}$ is not necessarily nonnegative, since $f$ and $|f|$ have same the $L^1$ and $L^\infty$ norms and
$$
u_{-|f|}\leq u_f\leq u_{|f|}
$$
then Theorem \ref{th:2.2} obviously holds.
\end{proof}


\subsection{Estimates on $\sigma_D$ for nonnegative $f$}

\begin{proof}[Proof of Theorem~\ref{th:2.3}]
The main tool in the proof is the following result of Talenti [\ref{bib:Talenti}].
\begin{lem}[Talenti's Lemma]
\label{talenti}
Let $f$ be a smooth nonnegative function in $D\subset \mathbb{R}^n$, and let $u$ solve the Dirichlet boundary value problem
\begin{equation}
\begin{cases} -\Delta u(x) = f (x)  & \mbox{in } D,
 \\
 u=0  & \mbox{on } \partial D. \end{cases}
\end{equation}
Assume $f^*$ and $u^*$ are the symmetric decreasing rearrangements of respectively $f$ and $u$ defined in a ball $B$, such that $|B|=|D|$. Then if $v$ solves the problem
\begin{equation}
\label{main}
\begin{cases} -\Delta v(x) = f^* (x)  & \mbox{in } B,
 \\
 v=0  & \mbox{on } \partial B,
 \end{cases}
\end{equation}
then
$$
u^*(x)\leq v(x),\quad{for all} x\in B.
$$
Moreover, the equality holds if and only if $D=B$ and $f=f^*$.
\end{lem}
We get the following simple corollary of Telenti's lemma.
\begin{lem}
\label{talenti.green}
Let $G^*_{\hat{x},D}(y)$ be the symmetric decreasing rearrangement of $G_D(\hat{x},y)$
with respect to the $y$ variable. Then
$$
G^*_{\hat{x},D}(y)\leq G_B(0,y)\quad \text{ for }\quad y\in B\setminus \{0\},
$$
where $G_B$ is the Green's function in the ball $B$ with $|B|=|D|$.
\end{lem}
\begin{proof}
Observe that
\begin{equation}
\begin{cases}
-\Delta_x G_{ D}(\hat{x},x) = \delta_{\hat{x}} (x)  & \mbox{in } D,
 \\
 G_{ D}(\hat{x},x)=0  & \mbox{for } x\in\partial D,
\end{cases}
\end{equation}
and
\begin{equation}
\label{main}
\begin{cases} -\Delta G_B(0,x) = \delta_0 (x)  & \mbox{in } B,
 \\
 G(0,x)=0  & \mbox{on } \partial B,
 \end{cases}
\end{equation}
where $\delta_y(x)$ is the Dirac measure in the point $y$. The proof would follow from Lemma \ref{talenti} if we could take $f$ a Dirac measure. This is a simple generalization which trivially follows from approximation $\frac{1}{\omega_nr^n}\chi_{B_r(y)}(x)\rightharpoonup \delta_y(x)$, we omit the details here.
\end{proof}
Next we obtain from (\ref{sigmadef}),
\begin{align*}
\sigma_D(\|f\|_1)&=u(\hat{x})\\
&=\int_D  G_D(\hat{x}, y) \chi_{\{G(\hat{x}, y)>\alpha_{\hat{x}}\}}dy\\
&=\int_B  G^*_{\hat{x},D}(y) \chi_{\{G^*_{\hat{x},D}( y)>\alpha_{\hat{x}}\}}dy\\
&\leq \int_B  G_0(0,y) \chi_{\{G_B(0, y)>\alpha_{0}\}}dy,
\end{align*}
where the last inequality follows from Lemma \ref{talenti.green} and the fact that
$${\{G^*_{\hat{x},D}( y)>\alpha_{\hat{x}}\}}= {\{G_B(0, y)>\alpha_{0}\}}.$$
To calculate the value of $\sigma_B$ we utilize the example (\ref{1.8}) in the ball $B_R:$
\begin{equation}
u_r(|x|)=
\begin{cases}
F(R)-F(|x|)  & \mbox{in } B_R\setminus B_r,\\
F(R)-F(r)+\frac{1}{2r^{n-2}}-\frac{1}{2r^n} |x|^2 & \mbox{in }  B_r,
\end{cases}
\end{equation}
where $F(|x|)$ is the fundamental solution multiplied by a dimensional constant such that $F'(r)= \frac{1}{r^{n-1}}$.
Computing the norms $\|f_r\|_{L^1(B)}$ and $\|f_r \|_{L^\infty(B)}$ of $f_r=-\Delta u_r$ we get
$$
\sigma_{B}(\omega_n r^n)=\frac{1}{n}r^n(F(R)-F(r))+\frac{1}{2}r^2,
$$
thus we obtain for $\sigma_B$ the formula
$$
\sigma_{B}(t)=\frac{t}{n\omega_n}(F(R)-F((t/\omega_n)^{1/n}))+\frac{1}{2}(t/\omega_n)^{2/n}.
$$
In particular for $n\geq 3$ we have $F(\tau)=-\frac{1}{(n-2)\tau^{n-2}}$ and
\begin{align*}
\sigma_{B}(t)&=\frac{t}{n(n-2)\omega_n}\left(\frac{1}{(t/\omega_n)^{(n-2)/n}}-\frac{1}{R^{n-2}}\right)+\frac{1}{2}(t/\omega_n)^{2/n}\\
&=C_1(n) t^{\frac{2}{n}}-C_2(n)\frac{t}{R^{n-2}},
\end{align*}
with
\begin{equation}
\label{constants}
C_1(n)=\frac{(n-1)^2+1}{2n(n-2)\omega_n^{2/n}}\quad
\text{ and }\quad
C_2(n)=\frac{1}{n(n-2)\omega_n}.
\end{equation}
For $n=2$ we need to take  $F(\tau)=\ln \tau$, $\omega_2=\pi,$ which yields
\begin{align*}
\sigma_{B}(t)&=\frac{t}{2\pi}\left(\ln R -\ln((t/\pi)^{1/2})\right)+\frac{1}{2}(t/\pi)\\
&=\left(\frac{1}{2}\ln \pi +\frac{1}{2\pi} +\frac{1}{2\pi}\ln R\right)t-\frac{1}{4\pi} t\ln t.
\end{align*}
The proof is finished now.

\end{proof}


\subsection{Signchanging Laplacian: The first approach.}

In the next two subsections we prove Theorem~\ref{th:2.4}. Note that (\ref{2.12}) is a finer estimate, which takes into account the cancelation phenomena. The first approach is to add the norm $\|f\|_\infty$ to the right hand side of the equation $-\Delta u=f,$ while the second one it to consider the positive and negative parts of $f,$ i.e., the functions $f^+$ and $f^-$.
From the linearity we know that
$$
u_f=u_g-\|f\|_{\infty} v
$$
where $g= f+\|f\|_{\infty}$ and $v=u_1$ is the solution of the Poisson problem with $f\equiv 1$. Observe that $g\geq 0$,  $\|g\|_{\infty}\leq 2\|f\|_{\infty}$
and $\|g\|_{L^1(D)}=I_f+\|f\|_{\infty}|D|$, where $I_f=\int_D fdx$.
Thus
$$
\|u_g\|_{\infty}\leq 2\|f\|_\infty \sigma_D\left( \frac{I_f+\|f\|_\infty|D|}{2\|f\|_\infty} \right),
$$
and
$$
\max_{x\in D} \,\,u(x)\leq \|f\|_\infty \left[ \sigma_D\left( \frac{1}{2}\left[\tfrac{1}{\|f\|_\infty}I_f + |D| \right]\right)-  v(\breve{x})\right],
$$
where $\breve{x}$ is the point where the maximum of $u$ is achieved.
Similarly we can consider the function $g=\|f\|_\infty -f$ and obtain
$$
\max_{x\in D} \,\,(-u(x))\leq \|f\|_\infty \left[ \sigma_D\left( \frac{1}{2}\left[-\tfrac{1}{\|f\|_\infty}I_f + |D| \right]\right)-  v(\tilde{x})\right],
$$
where $\tilde{x}$ is the point where the minimum of $u$ is achieved.
The last two estimates yield the bound (\ref{signchanging}).
\begin{rem}
The inequality (\ref{signchanging}) is interesting because of two reasons. First for a sign changing function $f$ it contains the integral $I_f$ of $f$ and not the norm $\|f\|_1.$ Second, the function $v$ is a known positive function depending only on the domain. If the location of the point where the maximum of $|u|$ is achieved can be estimated, then the term $-v(\hat{x})$ would provide significant improvement of the inequality.
\end{rem}

\subsection{The second approach: Proof of Theorem~\ref{th:2.4}}
Now we can write $f=f^++f^-$, and thus by linearity
$$
u_f=u_{f^+}+u_{f^-}.
$$
Since the functions on the right hand side have different signs, we get
$$
|u|\leq\max (u_{f^+},u_{f^-})
$$
and thus
\begin{equation}
\label{signchanging2}
\|u\|_\infty \leq\max \left\{ \|f^+\|_\infty \sigma_D(\|f^+\|_1/\|f^+\|_\infty)\,\, ,\,\, \|f^-\|_\infty \sigma_D(\|f^-\|_1/\|f^-\|_\infty)    \right\}.
\end{equation}
\begin{rem}
The inequality (\ref{signchanging2}) is optimal on non-connected domains. We only need to let the function $f$ have opposite signs on those domains. By connecting the disconnected domains by narrow tubes we will not change much the inequality.
\end{rem}
The estimates (\ref{2.13}) follow from Theorems~\ref{th:2.2}-\ref{th:2.3}.

\section{An application to the Laplace eigenfunctions}

In this section we derive an estimate on the $L^\infty$ norm of the $k-$th eigenfunction of the Laplace operator in the domain $D.$ Assume as usual that $(\lambda_k,u_k)$ is the $k-$th eigenvalue-eigenfunction pair of the Laplace operator, i.e.,
\begin{equation}
\label{4.1}
\begin{cases}
-\Delta u_k=\lambda_k u_k & \mbox{ in } \ \ D,\\
u_k=0 & \mbox{ on } \ \ \partial D.\\
\end{cases}
\end{equation}
We apply (\ref{2.10}) to the pair $(u_k,\lambda_ku_k)$ to derive an estimate on the norm $\|u_k\|_\infty$ in terms of the parameters $n,$ $\lambda_k$ and $|D|.$ Assume first $n>2,$ then we have by (\ref{2.12}) that
\begin{equation}
\label{4.2}
\frac{1}{\lambda_k}\|u_k\|_{\infty}\leq \frac{(n-1)^2+1}{2n(n-2)\omega_n^{2/n}}\|u_k\|_{1}^{\frac{2}{n}}\|u_k\|_{\infty}^{\frac{n-2}{n}}-
\frac{1}{n(n-2)\omega_nR^{n-2}}\|u_k\|_{1}.
\end{equation}
Let $\alpha,\beta>0$ be parameters yet to be chosen such that $\alpha^2\beta^{n-2}=1.$ Then we have by Young's inequality
$$
\|u_k\|_{1}^{\frac{2}{n}}\|u_k\|_{\infty}^{\frac{n-2}{n}}=(\alpha\|u_k\|_{1})^{\frac{2}{n}}(\beta\|u_k\|_{\infty})^{\frac{n-2}{n}}
\leq \frac{2\alpha}{n}\|u_k\|_{1}+\frac{\beta(n-2)}{n}\|u_k\|_{\infty},
$$
thus we get the estimate
\begin{equation}
\label{4.3}
\frac{(n-1)^2+1}{2n(n-2)\omega_n^{2/n}}\|u_k\|_{1}^{\frac{2}{n}}\|u_k\|_{\infty}^{\frac{n-2}{n}}
\leq \frac{\alpha((n-1)^2+1)}{n^2(n-2)\omega_n^{2/n}}\|u_k\|_{1}
+\frac{\beta((n-1)^2+1)}{2n^2\omega_n^{2/n}}\|u_k\|_{\infty}.
\end{equation}
We now choose $\beta$ such that $\frac{\beta((n-1)^2+1)}{2n^2\omega_n^{2/n}}=\frac{1}{2\lambda_k},$ which gives
\begin{equation}
\label{4.4}
\beta=\frac{n^2\omega_n^{2/n}}{\lambda_k((n-1)^2+1)},\quad \alpha=\frac{[\lambda_k((n-1)^2+1)]^{\frac{n-2}{2}}}{n^{n-2}\omega_n^{\frac{n-2}{n}}}.
\end{equation}
A combination of (\ref{4.2}), (\ref{4.3}) and (\ref{4.4}) gives (\ref{2.14}) for $n>2.$ For the proof for $n=2$ observe that if
$\|u_k\|_\infty\leq \|u_k\|_1,$ then (\ref{2.14}) follows from (\ref{2.12}). If else $\|u_k\|_\infty>\|u_k\|_1,$ then
one should utilize the inequality $\left|\|u_k\|_1\ln(\frac{\|u_k\|_\infty}{\|u_k\|_1})\right|\leq \sqrt{2\|u_k\|_1\|u_k\|_\infty}$ and then a suitable Young's inequality. We omit the details here.

\section*{Acknowledgements.}
The research of the second author is partially supported by the National Science Foundation of China (research grant nr. 11650110437).


\end{document}